\documentclass[12pt]{article}
\usepackage{epstopdf}
\usepackage{amsthm}
\usepackage{amsmath}
\usepackage{amssymb}

\newtheorem{thm}{Theorem}[section]

\theoremstyle{definition}
\newtheorem{defn}[thm]{Definition}
\newtheorem{lem}[thm]{Lemma}

\begin{document}

\title{
On the filtration of a free algebra by its associative lower central series.
}

\author{
George Kerchev
}


\date {}

\maketitle



\begin{abstract}

This paper concerns the associative lower central series ideals $M_i$ of the free 
algebra $A_n$ on $n$ generators.  Namely, we study the successive quotients 
$N_i=M_i/M_{i+1}$, which admit an action of the Lie algebra $W_n$ of vector fields on 
$\Bbb C^n$.  We bound the degree $|\lambda|$ of tensor field modules $F_\lambda$ appearing 
in the Jordan-H\"older series of each $N_i$, confirming a recent conjecture of 
Arbesfeld and Jordan.  As an application, we compute these decompositions for small 
$n$ and $i$.

\end{abstract}


\section{Introduction}
Let $A_n=\mathbb{C}\langle x_1,x_2,\ldots,x_n\rangle$ be the algebra over $\mathbb{C}$  of
noncommutative polynomials with generators $x_1, x_2, \ldots, x_n$. We consider the lower
central series of Lie ideals $L_i$ defined
inductively by $L_1=A_n$ and $L_{i+1}=[A_n, L_i]$. 
We denote by $M_i$ the two-sided ideal
in $A_n$ generated by $L_i$, $ M_i:=A_n L_i A_n$. This is
the same as the left-sided ideal $A_nL_i$. This follows from the
identity below where $a,c \in A_n$ and $b \in L_{i-1}$:
\begin{center}
$[a,b]c=-a[b,c]+[ac,b].$
\end{center}


In this paper, we study  the Jordan-H\"older series of
$N_i=M_i/M_{i+1}$.
The Jordan-H\"older series give decompositions (in the Groethendieck group) of $N_i$
into sums of irreducible $W_n$-modules of $\mathcal{F}_{\lambda}$, where
$\lambda=(\lambda_1,\lambda_2,\ldots,\lambda_n)$ for $\lambda_1\geq \lambda_2 \geq \ldots \geq \lambda_n$
are non-negative integers; $|\lambda|:=\lambda_1+\lambda_2+\ldots+\lambda_n$. Here $W_n$, the Lie algebra of polynomial vector fields, acts on the $N_i$. The Jordan-H\"older constituents are $F_\lambda$. 

We prove the following conjecture of Arbesfeld and Jordan on the upper bound of
$|\lambda|.$

\begin{thm}

For $\mathcal{F}_{\lambda}$ in the Jordan-H\"older series of $N_m$ we
have

\begin{center}
$|\lambda|\leq 2m-2 +2[\frac{n-2}{2}]$
\end{center}

For $m$ odd, this can be improved to
\begin{center}
 $|\lambda|\leq 2m-2.$
\end{center}

\label{C:1}
\end{thm}

We apply similar techniques to those of \cite{ND}, \cite{BJ}, who studied the Lie 
quotients $B_i=L_i/L_{i+1}$.  The proof of Theorem 1.1 depends on the following:

\begin{thm}
 $N_i=V\cdot L_i/(A\cdot L_{i+1} \cap V \cdot L_i)$ where $V$ is spanned by elements of
 $A_n$ of degree $\leq 1$ .
\label{T:1}
\end{thm}

Feigin and Shoikhet ~\cite{BB} introduced the quotients
$B_i=L_i/L_{i+1}$. They were further studied by  Dobrovolska, Etingof, Kim and  Ma ~\cite{ED,DK,PE:1}, as well as Arbesfeld and Jordan(~\cite{ND}). In ~\cite{ED}
and ~\cite{ND}, bounds $|\lambda|$ for the Jordan-H\"older series of $B_i$ were produced,
and checked for small $i$ and $n$ by computer. We prove an analogous bound
on $|\lambda|$ for the Jordan-H\"older series of $N_i$ and show that it
grows linearly with $i$.

The structure of this paper is as follows. The following two subesctions ~\ref{Pr-1} and ~\ref{Pr-2} contain a
review of the representation theory of the Lie algebra of polynomial vector fields and the tensor field modules over $W_n$.  Section~\ref{Rez} presents a proof of the main result.  In Section~\ref{JH}  the Jordan-H\"{o}lder series
for $N_m(A_n)$ for small $n$ and $m$ are computed.



\subsection{Representation theory of the Lie algebra of polynomial vector fields}{\label{Pr-1}}

Let $W_n$ denote the Lie algebra of polynomial vector fields. As a vector space, we have: 
$$W_n=\bigoplus_i\mathbb{C}[x_1,x_2,\ldots,x_n]\partial_i.$$
The Lie bracket is given by: 
$$[p\partial_i,q\partial_j]=p\frac{\partial q}{\partial  x_i}\partial_j-q\frac{\partial p}{\partial
  x_j}\partial_i$$.

According to [EKM], $W_n$ acts on each $N_i$, and as a $W_n$-module, $N_i$ has a Jordan-H\"older series whose simple quotients are of the form $\mathcal{F}_{\lambda}$ (see section 2.3 below for the definition). Let $h_M$ be the Hilbert series for a graded vector space $M$. Then we have:
$$h_{\mathcal{F}_{\lambda}}=\frac{p}{(1-t_1)(1-t_2)\cdots (1-t_n)}, $$ where $\deg{p}=\|\lambda\|.$


We will bound $h_{N_i}\prod_i(1-t_i)$.  This bound and the knowledge of the Hilbert series of $\mathcal{F}_\lambda$ as 
above
 will allow us to control decompositions of the $N_i$  into the $\mathcal{F}_{\lambda}$.

\subsection{Tensor field modules over $W_n$ }{\label{Pr-2}}

For $\lambda=(\lambda_1,\lambda_2,\ldots,\lambda_n)$, $\lambda_1\geq \lambda_2\geq \lambda_3 \ldots\geq
\lambda_n$, where the $\lambda_i$ are nonnegative integers, let $V_{\lambda}$ be the
irreducible representation of $\mathfrak{gl}_n $ of highest weight $\lambda$.
We set $|\lambda|:=\lambda_1+\lambda_2+\ldots+\lambda_n$.

Let $\mathcal{\breve F}_{\lambda}$ be the space of polynomial tensor fields of type $V_{\lambda}$ on $\mathbb{C}^n$. As a vector space $\mathcal{\breve F}_{\lambda}:=\mathbb{C}[x_1,x_2,\ldots ,x_n]\otimes V_{\lambda}$.
It is known that  $\mathcal{\breve F}_{\lambda}$ is a representation of $W_n$ with action given by the
standard Lie derivative formula for action of vector fields on covariant tensor fields (see [R]).

\begin{thm} ~\cite{R}  If $\lambda_1 \geq 2$, or if $\lambda = (1^n)$, then  $\mathcal{\breve F}_{\lambda}$ is irreducible. Otherwise,
if $\lambda = (1^k, 0^{n-k})$, then  $\mathcal{\breve F}_{\lambda}$ is the space
$\Omega ^k = \Omega ^ k(\mathbb{C}^n)$ of polynomial differential
$k$-forms on $\mathbb{C}^n$, and it contains a unique irreducible submodule which is the space
of all closed differential $k$-forms.
\end{thm}

Denote by $\mathcal{F}_{\lambda}$ the irreducible submodule of  $\mathcal{\breve F}_{\lambda}$, so that $\mathcal{F}_{\lambda} =  \mathcal{\breve F}_{\lambda}$ unless $\lambda = (1^k, 0^{n-k})$
for some $1 \leq k \leq n - 1$.

\section{Proof of Conjecture~\ref{C:1}}{\label{Rez}}

In this section we prove the bound of the
Jordan-H\"older series of $N_i$ stated in Theorem~\ref{C:1}.

\noindent We begin by proving Theorem 1.2.  This result is also useful in simplifying computation of the Hilbert
series of $N_i$.  We fix $n$, and let $A$ denote $A_n$.

\noindent For the proof of Theorem 1.2 we use the following Lemma:

\begin{lem} For $b\in L_{i-1}$, $a,x,y \in A$, we have the following identity:
\begin{center}
$ yx[a,b]=x[ya,b]+y[xa,b]-[xya,b] \mod A L_{i+1}.$
\end{center}

\label{EL}
\end{lem}

\begin{proof}
Let $b,a,x,y$ be as above. Then we have the following identities:

\begin{equation}
 [xa,b]=x[a,b]+a[x,b]  \mod L_{i+1},
\label{1}
\end{equation}
\begin{equation}
 [xya,b]=xy[a,b]+xa[y,b]+ya[x,b]  \mod L_{i+1}.
\label{2}
\end{equation}
Multiplying (1) by $y$, we get
\begin{equation}
 y[xa,b]=yx[a,b]+ya[x,b]   \mod A L_{i+1}.
\label{3}
\end{equation}
Interchanging $x$ and $y$ in (\ref{3}), we get
\begin{equation}
 x[ya,b]=xy[a,b]+xa[y,b]   \mod A L_{i+1}.
\label{4}
\end{equation}
Now subtract (\ref{2}) from (\ref{3}) and (\ref{4}). We get
\begin{equation}
 yx[a,b]=x[ya,b]+y[xa,b]-[xya,b] \mod A L_{i+1}.
\label{5}
\end{equation}
\end{proof}
Taking $y$ to be of degree one and applying the lemma repeatedly, we can reduce the term in front of the bracket of an arbitrary element of $A\cdot L_i$
to something in $V\cdot L_i$ by adding terms in $A \cdot L_{i+1} \cap V \cdot L_i$, which is exactly Theorem 1.2.

\subsection{Proof of Theorem 1.1}
We first recall some definitions and results that we will need in the proof.

\begin{defn} Let $\bar{Z}$ be the image of $A[A,[A,A]]$ in $B_1$,
which was shown in ~\cite{BB} to be central in $B$. We define $\bar B_1$ to be the quotient $\bar B_1=B/\bar{Z}$.
\end{defn}
Now, recall the Feigin-Shoiket map ~\cite{BB}.
\begin{thm}
\label{FSM}
There is a unique isomorphism of algebras,
$$\xi: \Omega^{ev}_{*} \to  A/A[A, [A,A]],$$
$$x_i \to x_i.$$
It restricts to an isomorphism $\xi: \Omega^{ev}_{*,ex} \to  B_2,$ and descends to an isomorphism
$\xi: \Omega^{ev}_{*}/\Omega^{ev}_{*,ex} \to  \bar B_1,$.
\end{thm}

For the second part of the proof of Theorem~\ref{C:1} we use some of the methods
introduced in~\cite{ND}. Recall the map from~\cite{ND} :
\begin{thm}
\label{mapth}
There is a surjective map $f_m: (\Omega^{ev})^{\otimes m} \to B_m$ such that $$f_m(a_1, \ldots , a_m)=[\xi(a_1),[\xi(a_2), \ldots [\xi(a_{m-1}),\xi(a_m)]]],$$
where $\xi: \Omega^{ev} \to \breve B_1$ is the Feigin-Shoiket map from Theorem ~\ref{FSM}.
\end{thm}

We recall that the algebra of zero forms is in fact $\mathbb{C}[x_1,x_2,\ldots,x_n]$ and
will be referred to as $S$.
We also use that $f_m$ is surjective when restricted to
$Y:=(\Omega^0)^{\otimes m-2}\otimes (\bigoplus_{j+k\leq\left[\frac{n-2}{2}\right]}\Omega^{2j}\otimes\Omega^{2k})$ as shown by Bapat and Jordan.
Using Theorem 1.2, we
define a similar map for $$Z=S\otimes
Y=\Omega^0\otimes(\Omega^0)^{\otimes m-2}\otimes(\bigoplus_{j+k\leq\left[\frac{n-2}{2}\right]}\Omega^{2j}\otimes\Omega^{2k}).$$
Since $f_m|_{Y}$ is surjective, by Theorem 1.2, so is
$$\breve f_m : Z \twoheadrightarrow N_m.$$
$$a\otimes b \mapsto af_m(b)$$
where $a \in \Omega^0$, $b \in \Omega^0.$ Here, on the right hand side, by abuse of notation, $a$ is $\xi(a).$

Surjectivity of $\breve f_m$ implies that the Jordan-H\"older series of $Z$ dominates
the Jordan-H\"older series of $N_m$. We seek a large $W_n$-submodule $\breve I\subseteq\text{Ker } \breve f_m,$
  so that the Jordan-H\"older series of $Z/\breve I$ still
 dominates the Jordan-H\"older series of $N_m$. Then for the proof of
Conjecture 1.1 it will be sufficient to show that all
 $\mathcal{F}_{\lambda}$ occurring in the Jordan-H\"older series of $Z/\breve I$ satisfy the
bound on $|\lambda|$.

As in ~\cite{ND}, we define $R$ as
$$R:=\mathbb{C}[x_1,x_2,\ldots,x_n]^{\otimes m}=\mathbb{C}[x_{1,1},x_{2,1},\ldots,
x_{n,1},x_{1,2},\ldots ,x_{n,m}].$$ Let $R^{\prime}=S\otimes
R=\mathbb{C}[x_{1,0},x_{2,0},\ldots,
x_{n,0},x_{1,1},\ldots ,x_{n,m}]$. Let the ideals $J_j$ of $R^{\prime}$
for $0\leq j \leq m-1$ be generated by
$X_{i,j}=x_{i,j}-x_{i,j+1}.$

Let $I = J_0^2+ \sum_{i=1}^{m-2} J_i^3 + J_{m-1}^2.$
Let $J = \sum_{i=1}^{m-2} J_i^3 + J_{m-1}^2.$ We will show that $\breve
I=IZ$ is in Ker $\breve f_m$

\begin{lem} The ideal $J_0^2Z$ is a subset of the Kernel of $ \breve f_m$
\label{polez}
\end{lem}

\begin{proof}
This is straightforward from Lemma~\ref{EL}.  $J_0^2Z$ is spanned by
elements of the form $(x \otimes 1 \otimes 1 - 1 \otimes x \otimes 1)*(y \otimes 1
  \otimes 1-1 \otimes y \otimes 1)*(1\otimes a
  \otimes b)$ where $*$ is the Fedosov product. We have that
$$(x \otimes 1 \otimes 1 - 1 \otimes x \otimes 1)*(y \otimes 1
  \otimes 1-1 \otimes y \otimes 1)*(1\otimes a
  \otimes b) =$$
$$xy \otimes a \otimes b - x \otimes ya \otimes b - y
  \otimes xa \otimes b + 1 \otimes xya \otimes b. $$
Consider the image of the map $\breve f_m$.  $$\breve f_m(xy \otimes a \otimes b - x \otimes ya \otimes b - y
  \otimes xa \otimes b + 1 \otimes xya \otimes b)=
yx[a,b]-x[ya,b]+y[xa,b]-[xya,b].$$ By Lemma~\ref{EL} this is zero in
$N_m$ so the spanning set of $J_0^2Z$ maps to zero.
\end{proof}

Arbesfeld and Jordan showed that $JZ\subseteq \text{Ker}
f_m$. Lemma~\ref{polez} implies that $J_0^2\cdot Z\subseteq
\text{Ker}\breve f_m$. Thus we have $Z/IZ\twoheadrightarrow N_m$.






We finish the proof analogously to~\cite{ND}.

As in~\cite{ND} $h_{Z/IZ}=h_{R^{\prime}/I}\times h_{X^{\prime}}$ where
$h_{X^{\prime}}$ is the Hilbert series of the generators over
$R^{\prime}$ of $Z$ . Again by using the results from~\cite{BJ} we have
$$ h_{X^{\prime}}=\sum_{j+k\leq 2\lfloor\frac{n-2}{2}\rfloor}\sigma_{2j}\times\sigma_{2k}.$$
\noindent where $\sigma_l=\sum_{i_1\leq i_2\leq \ldots \leq
  i_l}t_{i_1}t_{i_2}\cdots t_{i_l}$ are e the elementary symmetric
functions. We can also compute
$$h_{R^{\prime}/\breve J}=\frac{(1+\sum t_i + \sum_{i\leq j}t_i
  t_j)^{m-2}(1+\sum t_i)^2}{(1-t_1)(1-t_2)\cdots (1-t_n)}.$$
We now use that $$h_{Z/IZ}=
\frac{Q(t_1,t_2,\ldots, t_n)}{(1-t_1)(1-t_2)\cdots (1-t_n)}.$$
Thus $Q= h_{X^{\prime}}\times (1+\sum t_i + \sum_{i\leq j}t_i
  t_j)^{m-2}(1+\sum t_i)^2.$
So the degree of $Q$ is $2m-2 +2\lfloor\frac{n-2}{2}\rfloor$.

For $m$ odd, we can improve this bound to $|\lambda|\leq 2m-2$ using the following result of [BJ]:

\begin{thm}
\label{Mjk}
$M_j M_k\subset M_{j+k-1}$ whenever j or k is odd.
\end{thm}

To apply this, we use the following argument suggested by Pavel Etingof.  Notice that
$$a[a_1,...,[a_{m-1},b[c,d]]] = a\sum_{S\subset [1, m-1] }(\prod_{i \in S} {\rm ad} a_i)(b) \cdot (\prod_{i \notin S}{\rm ad}a_l)([c,d]))$$

If $[S] = s$, then the corresponding term on the right hand side is in $M_{s+1}M_{m-s+1}$. But one of the numbers $s+1$ and $m-s+1$ is odd, since their sum is $m+2$ which is odd. So by theorem 2.6, all the terms on the right hand side are in $M_{m+1}$, hence are zero in $N_m$. So the left hand side is zero in $M_m$.
Now under the Feigin-Shoiket isomophism, $A[A,A]$ corresponds to forms of degree $2$ and higher, so in the proof of Theorem $\ref{mapth}$, we may replace $\bigoplus_{j+k\leq\left[\frac{n-2}{2}\right]}\Omega^{2j}\otimes\Omega^{2k}$ by $\Omega^0$.  We may then carry out the argument exactly as above, but with the dependence on $n$ removed.







\section{Conclusion}{\label{JH}}
Now that we have found a lower bound of $|\lambda|$ (Theorem~\ref{C:1}), we
may obtain the Jordan-H\"older series for several values of $m,n$ for
$N_m(A_n)$ via MAGMA computation.  For example, we have the following
results, in which we will denote each instance of
$\mathcal{F}_{\lambda}$ by the $n$-tuple $\lambda$ for economy of notation.

\begin{thm}
The Jordan-H\"older series for $N_m(A_2)$ for $3 \leq m \leq 7$ are:
\begin{itemize}
\item $N_3=(2, 1) + (2, 2)$.
\item $N_4=(3, 1) + (3, 2) + (3, 3)$.
\item $N_5=(4, 1) + (3, 2) + 2(4, 2) + (4, 3) + (4, 4)$.
\item $N_6=(5, 1) + (4, 2) + (3, 3) + 2(5, 2) + 2(4, 3) + 2(5, 3) + (5,
  4) + (5, 5)$.
\item $N_7=  (6, 1) + 2(5, 2) + 2(4, 3) + 3(6, 2) + 3(5, 3) + 3(4,
  4) + 3(6, 3) + 2(5, 4) + 2(6, 4) + (6, 5) + (6, 6)$.
\end{itemize}
\end{thm}
These decompositions were conjectured by Arbesfeld and are now theorems.
\begin{thm}
The Jordan-H\"older series for $N_m(A_3)$ for $m=3, 4$ are:
\begin{itemize}
\item $N_3=(2,1,0) + (2,2,0)$.
\item $N_4=(2, 2, 2) + (2, 2,1) + (3, 1,0)+(3,1,1)+(3,2,0)+(3,3,0)$.
\end{itemize}
\end{thm}

\begin{thm}
The Jordan-H\"older series for $N_m(A_4)$, $m=3, 4$ are:
\begin{itemize}
\item $N_3=(2,1,0,0)+(2,2,0,0)$.
\item $N_4=(3,3,0,0) + (3,2,0,0) + (3,1,1,1) + (3,1,1,0) + (3,1,0,0) + (2,2,1,1) + (2,2,2,0) + (2,1,1,1) + (2,1,1,0). $

\end{itemize}
\end{thm}


Note that the decomposition of $N_3$ was also computed in [EKM].

\section{Acknowledgments}

This paper is the result of research done at the Reserach Science Institute at MIT. The author would like to express his gratitude to Bhairav Singh for his helpful contribution as a mentor during RSI, to  Pavel Etingof for posing the problem, and also to David
Jordan, Martina Balagovic, Asilata Bapat for the useful discussions on Lie
algebras, representation theory and MAGMA. The  author would like to thank to St.~Cyril and
St.~Methodius International Foundation, DecArt, the High Student Institute of Mathematics and
Informatics of the Bulgarian Academy of Sciences, CEE, RSI and MIT for funding.




\end{document}